\documentclass[a4paper,12pt]{article}

\usepackage{amssymb}
\usepackage{latexsym}
\usepackage{amsfonts}
\usepackage{amsmath}
\usepackage{amsthm}
\usepackage{hyperref}
\usepackage{tensor}
\usepackage{mathtools}
\usepackage{mdwlist}
\usepackage{pgf,tikz,pgfplots}
\pgfplotsset{width=7.5cm,compat=1.9}

\usetikzlibrary{arrows,patterns,positioning}
\usetikzlibrary{calc}

\usepackage[margin=1in]{geometry}

\usepackage[compact]{titlesec}

\usepackage[scale=0.9]{tgheros}
\usepackage[T1]{fontenc}

\DeclareMathOperator{\SL}{SL}

\DeclareMathOperator{\PGL}{PGL}

\DeclareMathOperator{\pg}{PG}

\DeclareMathOperator{\PGaL}{P\Gamma L}

\DeclareMathOperator{\tr}{Tr}

\DeclareMathOperator{\alt}{A}

\renewcommand{\leq}{\leqslant}
\renewcommand{\geq}{\geqslant}

\newcommand{\F}{\mathbb F}

\newcommand{\B}{\mathcal B}

\renewcommand{\P}{\mathcal P}

\theoremstyle{plain}
\newtheorem{lemma}{Lemma}
\newtheorem{theorem}[lemma]{Theorem}

\theoremstyle{definition}

\numberwithin{equation}{section}
\numberwithin{lemma}{section}

\begin{document}

\title{Transitive $(q-1)$-fold packings of $\pg_n(q)$}
\author{Daniel R. Hawtin\footnote{Address: Faculty of Mathematics, University of Rijeka. Rijeka 51000. Croatia. Email: \href{mailto:dhawtin@math.uniri.hr}{dhawtin@math.uniri.hr}}}

\maketitle
\begin{abstract}
 A \emph{$t$-fold packing} of a projective space $\pg_n(q)$ is a collection $\P$ of line-spreads such that each line of $\pg_n(q)$ occurs in precisely $t$ spreads in $\P$. A $t$-fold packing $\P$ is \emph{transitive} if a subgroup of $\PGaL_{n+1}(q)$ preserves and acts transitively on $\P$. We give a construction for a transitive $(q-1)$-fold packing of $\pg_n(q)$, where $q=2^k$, for any odd positive integers $n$ and $k$, such that $n\geq 3$. This generalises a construction of Baker from 1976 for the case $q=2$.
\end{abstract}

\section{Introduction}

In 1976, Baker \cite{baker1976partitioning} constructed a partition of the set of lines of the projective geometry $\pg_n(2)$ into spreads. (Note that in this paper by a \emph{spread} we mean a set of lines of a projective space inducing a partition of the point-set of the geometry.) Such a construction is often refered to as a \emph{parallelism} or \emph{packing}. Some applications of Baker's construction include: a description of the Preparata codes \cite{baker1983preparata}, the construction of a family of antipodal distance-regular graphs \cite{de1995family}, and the determination of the chromatic number of the Grassmann graph $J_q(n+1,2)$ (see \cite[Section~3.5.1]{brouwer2012strongly}).

As well as Baker's construction, there are several other results results concerning infinite families of packings in projective spaces. In particular, Denniston \cite{denniston1972some} proved that packings exist in $\pg_3(q)$ for all prime powers $q$. Moreover, Penttila and Williams \cite{penttila1998regular} gave a construction for two inequivalent regular packings of $\pg_3(q)$ for each $q\equiv {2 \pmod 3}$, where a packing is regular if each of its constituent spreads are regular (Desarguesian). In higher dimensions, Beutelspacher \cite{beutelspacher1974parallelisms} showed that packings exist in $\pg_n(q)$ if $n=2^{i+1}-1$ and $i$ is any positive integer. See \cite{johnson2010combinatorics} for a fairly comprehensive survey of packings in projective spaces.

As a generalisation of a packing, we are interested here in the concept of a \emph{$t$-fold packing}, defined to be a collection $\P$ of spreads of $\pg_n(q)$ such that every line is contained in precisely $t$ elements of $\P$. Further, we say that a $t$-fold packing $\P$ is \emph{transitive} if there exists a subgroup of $\PGaL_n(q)$ leaving $\P$ invariant and acting transitively on the spreads of $\P$. The main result of this paper, stated below, is a generalisation of \cite{baker1976partitioning}, giving Baker's construction when $q=2$. 

\begin{theorem}\label{thm:main}
 Let $n$ and $k$ be odd positive integers, with $n\geq 3$, let $q=2^k$, and let $\P$ be as in (\ref{eq:partition}). Then $\P$ is a transitive $(q-1)$-fold packing of $\pg_n(q)$.
\end{theorem}

The only currently known $t$-fold packing of $\pg_n(q)$ that the author is aware of is a $5$-fold packing of $\pg_3(2)$, a description of which can be found in a comment of John Bamberg in a SymOmega blog post\footnote{\href{https://symomega.wordpress.com/2009/09/11/i-want-more-moore-graphs/}{https://symomega.wordpress.com/2009/09/11/i-want-more-moore-graphs/}} (in fact, the current author has taken the terminology ``$t$-fold packing'' from said comment). The $5$-fold packing is given by one of the orbits of $\alt_7\leq \PGL_4(2)$ on the spreads of $\pg_3(2)$.

\section{The construction}

Let $k$ and $n$ be odd positive integers, with $n\geq 3$, and let $q=2^k$. Let $U=\F_{q^n}$, viewed as an $\F_q$-vector space of rank $n$, let $W=\langle w\rangle\cong\F_q$, with $w\notin U$, and let $V=U\oplus W\cong \F_q^{n+1}$. We take $V$ to be the underlying vector space of the projective geometry $\pg_n(q)$. The notation $\F_{q}^\times$ refers to both the set of non-zero elements of the field and the multiplicative group of the field.

Let $\alpha\in \F_{q^n}^\times$. Then we define $\B_\alpha$ to be the set of all those lines $\ell$ of $\pg_n(q)$ such that there exists a basis $x+x_0w,y+y_0w$ for $\ell$, where $x,y\in \F_{q^n}^\times$ and $x_0,y_0\in\F_q$, for which
\begin{equation}\label{eq:main}
  \begin{vmatrix}
   x & x^q \\
   y & y^q
  \end{vmatrix}
  +
  \begin{vmatrix}
   x & x_0 \\
   y & y_0
  \end{vmatrix}^{q+1}
  =\alpha. 
\end{equation}
Finally, we define
\begin{equation}\label{eq:partition}
 \P=\{\B_\alpha\mid \alpha\in\F_{q^n}^\times\}.
\end{equation}
Note that the value of left hand side of (\ref{eq:main}) depends on the particular choice of basis for $\ell$. However, the next result shows that (\ref{eq:main}) is invariant under a determinant $1$ change of basis of $\ell$.

\begin{lemma}\label{lem:SL2qInvariant}
 For each $A\in\SL_2(q)$, the equation (\ref{eq:main}) is invariant under the map 
 \[
  \begin{bmatrix}
   x+x_0w \\
   y+y_0w
  \end{bmatrix}
  \mapsto A
  \begin{bmatrix}
   x+x_0w \\
   y+y_0w
  \end{bmatrix},
 \]
 on the basis for $\ell$.
\end{lemma}

\begin{proof}
 Since $\det A=1$, we have that
 \[
  \det \left(A
  \begin{bmatrix}
   x & x^q \\
   y & y^q
  \end{bmatrix}\right)
  +
  \det \left(A
  \begin{bmatrix}
   x & x_0 \\
   y & y_0
  \end{bmatrix}\right)^{q+1}
  =
  \begin{vmatrix}
   x & x^q \\
   y & y^q
  \end{vmatrix}
  +
  \begin{vmatrix}
   x & x_0 \\
   y & y_0
  \end{vmatrix}^{q+1},
 \]
 and the result holds.
\end{proof}

The next three lemmas show, for each $\alpha\in\F_{q^n}^\times$, that the set $\B_\alpha$ is in fact a spread, the first being required to prove Lemma~\ref{lem:Ucontained}.

\begin{lemma}\label{lem:roots}
 For any $\alpha,u\in \F_{q^n}^\times$, there exists precisely one $\lambda\in\F_q$ such that $ux^q+ u^qx+\lambda u^{q+1}+\alpha$ has roots for $x$ in $\F_{q^n}$.
\end{lemma}

\begin{proof}
 Replacing $x$ by $ux$ and dividing by $u^{q+1}$ we obtain the polynomial
 \[
  x^q+x+\lambda +\alpha u^{-(q+1)}.
 \]
 The polynomial $f(x)=x^q+x$ is a linearised polynomial with set of roots precisely $\F_q$. Thus the image $Y$ of $f$ is a codimension $1$ subspace of $U$ over $\F_q$. Since $n$ and $k$ are both odd, we have that $\tr(1)=1$, where $\tr$ is the absolute trace function $\F_{q^n}\rightarrow\F_2$. Applying \cite[Corollary 3.79 and Theorem 3.80]{lidl1997finite}, it follows that $x^q+x+1$ factors into $q/2$ irreducible polynomials, each of degree $2$, and hence $1\notin Y$. Thus the set $\{Y+\lambda\mid\lambda\in\F_q\}$ of cosets of $Y$ forms a partition of $U$ and $\alpha u^{-(q+1)}$ lies in precisely one such coset. Thus the result holds.
\end{proof}

\begin{lemma}\label{lem:Ucontained}
 If $u,\alpha\in \F_{q^n}^\times$ then $u$ is contained in a unique element of $\B_\alpha$.
\end{lemma}

\begin{proof}
 Setting $x=u$ and $x_0=0$ in Equation~\ref{eq:main}, and noting that $y_0^q=y_0$, we consider the following:
 \[
  uy^q+u^qy+u^{q+1}y_0^2=\alpha.
 \]
 By Lemma~\ref{lem:roots}, there is precisely one value of $y_0\in \F_q$ such that this equation has solutions for $y\in\F_{q^n}$. If $(y,y_0)=(v,v_0)$ gives one solution, then $(y,y_0)=(\lambda u+v,v_0)$, for $\lambda\in\F_q^\times$, give the remaining $q-1$. It follows that $\ell=\langle u, v+v_0w\rangle$ is the unique element of $\B_\alpha$ containing $u$.
\end{proof}

\begin{lemma}\label{lem:UwContained}
 Let $\alpha\in \F_{q^n}^\times$ and let $u\in \F_{q^n}$. Then $u+w$ is contained in precisely one element of $\B_\alpha$.
\end{lemma}

\begin{proof}
 Setting $x=u,x_0=1$ and $y_0=0$ in Equation~(\ref{eq:main}) gives
 \[
  uy^q+u^qy+y^{q+1}=\alpha.
 \]
 Since $(u+y)^{q+1}=u^{q+1}+uy^q+u^qy+y^{q+1}$, it follows that $(u+y)^{q+1}=u^{q+1}+\alpha$. The fact that $n$ is odd implies that $q+1$ is coprime to $q^n-1$. It follows that $(q+1)$-st roots are unique in $\F_{q^n}$, and we have that there is a unique $y\in\F_{q^n}^\times$ satisfying the above equation, given by
 \[
  y=\left( u^{q+1}+\alpha\right)^{1/(q+1)}+u.
 \]
 This completes the proof. 
\end{proof}

The following three lemmas demonstrate that $\P$ is a transitive $(q-1)$-fold packing. 

\begin{lemma}\label{lem:nonzero}
 If $x+x_0w,y+y_0w$ is a basis for a line $\ell$ of $\pg_n(q)$ then 
 \[
  \begin{vmatrix}
   x & x^q \\
   y & y^q
  \end{vmatrix}
  +
  \begin{vmatrix}
   x & x_0 \\
   y & y_0
  \end{vmatrix}^{q+1}\neq 0.
 \]
\end{lemma}

\begin{proof}
 If  By Lemma~\ref{lem:SL2qInvariant}, we can assume that $x_0=0$ and that $y_0=0$ or $1$. If $y_0=0$ then, since $xy^q+x^qy=0$ if and only if $x$ and $y$ are linearly dependent, the result holds in this case. Suppose $y_0=1$ and that the above inequality is instead an equality. Then we have 
 \[
  xy^q+x^qy+x^{q+1}=0.
 \]
 Since $(x+y)^{q+1}=y^{q+1}+xy^q+x^qy+x^{q+1}$ for all $x,y\in\F_{q^n}$, the above becomes $(x+y)^{q+1}=y^{q+1}$. However, since $n$ is odd, implying that $q+1$ is coprime to $q^n-1$ and $(q+1)$-st roots are unique in $\F_{q^n}$, this implies that $x=0$, which contradicts the assumptions on the basis for $\ell$. Thus the result holds.
\end{proof}

\begin{lemma}\label{lem:qMinus1fold}
 Every line of $\pg_n(q)$ is contained in $q-1$ elements of $\P$.
\end{lemma}

\begin{proof}
 Let $\ell$ be a line of $\pg_n(q)$ and $x+x_0w, y+y_0w$ be a basis for $\ell$, where $x,y\in \F_{q^n}^\times$ and $x_0,y_0\in\F_q$. By Lemma~\ref{lem:nonzero}, the left hand side of (\ref{eq:main}) is never $0$. By Lemma~\ref{lem:SL2qInvariant}, it suffices for us to show that
 \[
  \begin{vmatrix}
   \lambda x & \lambda x^q \\
   y & y^q
  \end{vmatrix}
  +
  \begin{vmatrix}
   \lambda x & \lambda x_0 \\
   y & y_0
  \end{vmatrix}^{q+1}
  \neq 
  \begin{vmatrix}
   x & x^q \\
   y & y^q
  \end{vmatrix}
  +
  \begin{vmatrix}
   x & x_0 \\
   y & y_0
  \end{vmatrix}^{q+1}
 \] 
 for any $\lambda\in\F_q\setminus\{0,1\}$. Suppose, to the contrary, that equality holds in the previous equation. Then, since $\lambda^2+1=(\lambda+1)^2$, we can rearrange the above to give
 \[
  (\lambda+1)
  \begin{vmatrix}
   x & x^q \\
   y & y^q
  \end{vmatrix}
  +
  (\lambda+1)^2
  \begin{vmatrix}
   x & x_0 \\
   y & y_0
  \end{vmatrix}^{q+1}
  =
  \begin{vmatrix}
   (\lambda+1) x & (\lambda+1) x^q \\
   y & y^q
  \end{vmatrix}
  +
  \begin{vmatrix}
   (\lambda+1) x & (\lambda+1) x_0 \\
   y & y_0
  \end{vmatrix}^{q+1}
  =0.
 \] 
 Since $(\lambda+1)(x+x_0w),y+y_0w$ is also a basis for $\ell$, this contradicts Lemma~\ref{lem:nonzero}, completing the proof.
\end{proof}

In the next lemma we consider the action of the multiplicative group $\F_{q^n}^\times$ on $V$ given by $(x+x_0w)^\beta=\beta x+x_0w$, where $\beta\in\F_{q^n}^\times$, $x\in\F_{q^n}$ and $x_0\in\F_q$.

\begin{lemma}\label{lem:transitive}
 The action of the multiplicative group $\F_{q^n}^\times$ on $U$ induces a transitive action on $\P$.
\end{lemma}

\begin{proof}
 Let $\alpha,\beta\in\F_{q^n}^\times$, let $\ell\in\B_\alpha$ and let $x+x_0w, y+y_0w$ be a basis for $\ell$ such that (\ref{eq:main}) holds, where $x,y\in \F_{q^n}^\times$ and $x_0,y_0\in\F_q$. Under the map $(x,x_0,y,y_0)\mapsto (\beta x,x_0,\beta y,y_0)$ the left hand side of (\ref{eq:main}) becomes
 \[
  \begin{vmatrix}
   \beta x & (\beta x)^q \\
   \beta y & (\beta y)^q
  \end{vmatrix}
  +
  \begin{vmatrix}
   \beta x & x_0 \\
   \beta y & y_0
  \end{vmatrix}^{q+1}
  =
  \beta^{q+1}
  \begin{vmatrix}
   x & x^q \\
   y & y^q
  \end{vmatrix}
  +
  \beta^{q+1}
  \begin{vmatrix}
   x & x_0 \\
   y & y_0
  \end{vmatrix}^{q+1}  
  =\alpha\beta^{q+1}.
 \]
 Hence $\ell^\beta$ is in $\B_\gamma$, where $\gamma=\alpha\beta^{q+1}$. Since $n$ and $k$ are odd, it follows that $q+1$ and $q^n-1$ are coprime and $\gamma$ ranges over all values of $\F_{q^n}^\times$ as $\beta$ does. Thus the result holds.
\end{proof}

We now prove the main theorem.

\begin{proof}[Proof of Theorem~\ref{thm:main}]
 Lemmas~\ref{lem:Ucontained} and~\ref{lem:UwContained} show that $\B_\alpha$ is a spread for each $\alpha\in\F_{q^n}^\times$. By Lemma~\ref{lem:qMinus1fold}, we have that $\P$ is a $(q-1)$-fold packing. Finally, Lemma~\ref{lem:transitive} shows that $\P$ is transitive.
\end{proof}


\end{document}